\documentclass[12pt,a4paper,reqno]{amsart}
\usepackage{cleveref}
\usepackage[abbrev]{amsrefs}

\parskip\medskipamount

\newtheorem{thm}{Theorem}[section]

\newtheorem{lem}[thm]{Lemma}
\newtheorem{prop}[thm]{Proposition}

\theoremstyle{definition}
\newtheorem{defn}[thm]{Definition}

\DeclareMathOperator{\card}{card}
\DeclareMathOperator{\interior}{int}
\DeclareMathOperator{\Lip}{Lip}
\DeclareMathOperator{\supp}{supp}

\newcommand{\ce}[1]{\mathcal{E}(#1)}
\newcommand{\free}[1]{\ensuremath{\mathcal{F}({#1})}}
\newcommand{\inter}[1]{\interior(#1)}
\newcommand{\ip}[2]{\ensuremath{\left\langle{#1}\,|\,{#2}\right\rangle}}
\newcommand{\lint}[4]{\ensuremath{\int_{#1}^{#2}{#3}\:\mathrm{d}{#4}}}
\newcommand{\Lipo}[1]{\Lip_0(#1)}
\newcommand{\N}{\mathbb{N}}

\newcommand{\nm}[1]{\left\vert\kern-0.25ex\left\vert #1 \right\vert\kern-0.25ex\right\vert}
\newcommand{\nmt}[1]{{\left\vert\kern-0.25ex\left\vert\kern-0.25ex\left\vert #1 \right\vert\kern-0.25ex\right\vert\kern-0.25ex\right\vert}}
\newcommand{\R}{\mathbb{R}}
\newcommand{\ts}{\textstyle}

\numberwithin{equation}{section}

\allowdisplaybreaks

\begin{document}
	\title{Lipschitz-free spaces over manifolds and the Metric Approximation Property}
	\author{Richard J.~Smith}
	\address{School of Mathematics and Statistics, University College Dublin, Belfield, Dublin 4, Ireland}
	\email{richard.smith@maths.ucd.ie}
	\author{Filip Talimdjioski}
	\address{School of Mathematics and Statistics, University College Dublin, Belfield, Dublin 4, Ireland}
	\email{filip.talimdjioski@ucdconnect.ie}
	
	\begin{abstract}
		Let $\nm\cdot$ be a norm on $\R^N$ and let $M$ be a closed $C^1$-submanifold of $\R^N$. Consider the pointed metric space $(M,d)$, where $d$ is the metric given by $d(x,y)=\nm{x-y}$, $x,y\in M$. Then the Lipschitz-free space $\free{M}$ has the Metric Approximation Property.
	\end{abstract}
	
	
	\date{\today}
	\maketitle
	
	\section{Introduction}\label{sect_intro}
	
	For a metric space $(M,d)$ and a point $x_0\in M$ let $\Lipo{M,x_0}$ be the Banach space of all real-valued Lipschitz functions $f$ on $M$ vanishing at $x_0$, equipped with the norm $$\nm{f} = \sup \left\{ \frac{|f(x)-f(y)|}{d(x,y)} : x,y\in M, x\not=y \right\}.$$ For every $x\in M$ let $\delta_x \in \Lipo{M,x_0}^*$ be the bounded functional defined by $\delta_x(f) = f(x)$ for all $f$. The Lipschitz-free space $\free{M}$ over $M$ is defined to be the closed linear span of $\{\delta_x : x\in M\} \subseteq \Lipo{M,x_0}^*$. The map $\delta\colon M\to \free{M}, \delta(x) = \delta_x$ is an injective isometry of $M$ into $\free{M}$. The free space $\free{M}$ is an isometric predual to $\Lipo{M,x_0}$ and it satisfies the fundamental property that any Lipschitz map $F$ from $M$ to another pointed metric space $N$, preserving the base point, uniquely extends, via the map $\delta$, to a bounded linear map between $\free{M}$ and $\free{N}$ whose norm equals the Lipschitz constant of $F$. This can be used as a tool for linearization of Lipschitz maps and therefore transferring nonlinear problems to a linear setting. The book \cite{weaver} by Weaver is devoted to Lipschitz spaces and investigates their Banach space, lattice and algebraic structure, among other topics.
	
	One of the directions of research has been approximation properties of Lipschitz-free spaces. In their seminal paper \cite{godefroykalton}, Godefroy and Kalton showed that if there exists a Lipschitz bijection between $X$ and $Y$ with a Lipschitz inverse, and $X$ has the bounded approximation property (BAP), then $Y$ has the BAP. To this end, they prove that $X$ has the $\lambda$-bounded approximation property ($\lambda$-BAP) if and only if $\free{X}$ has the $\lambda$-BAP, and that free spaces over finite-dimensional Banach spaces have the metric approximation property (MAP). Many results have since appeared that prove the ($\lambda$-bounded) approximation property of free spaces over various classes of metric spaces. We present a short survey of them.
	
	Regarding compact metric spaces, it is known that if $M$ is a sufficiently `thin' totally disconnected metric space then $\free{M}$ has the MAP. More precisely, if $M$ is a countable metric space such that its closed balls are compact, then $\free{M}$ has the MAP \cite{adalet}. Also, $\free{M}$ has the MAP when $M$ is compact and uniformly disconnected \cite{weaver}*{Corollary 4.39}, meaning that there exists $0 < a \leq 1$ such that for every distinct $p,q\in M$ there exist complementary clopen sets $C,D \subseteq M$ such that $p\in C, q\in D$ and $d(C,D) \geq a d(p,q)$, where $d(C,D) = \inf \{d(x,y) : x\in C, y\in D\}$. For example, the middle-third Cantor set $C$ and the Cantor dust $C^2 \subseteq \R^2$ satisfy the previous hypotheses. On the other hand, there exists a compact convex subset $K$ of a Banach space such that $\free{K}$ fails the approximation property (AP) \cite{godefroyozawa}. Moreover, in \cite{hajeklancienpernecka} it is shown that there exist metric spaces homeomorphic to the Cantor set whose free spaces fail the AP. The recent result \cite{douchakaufmann} shows that free spaces over compact groups equipped with a left-invariant metric have the MAP. A characterisation of the BAP on free spaces in terms of uniformly bounded sequences of Lipschitz `almost extension' operators is given in \cite{godefroy:15}.
	
	Apart from the compact case, it is known that if $M$ is a separable ultrametric space then $\free{M}$ has a monotone Schauder basis (and hence the MAP), and is linearly isomorphic to $l^1$ \cite{cuthdoucha}. Also, the free space over the Urysohn space has the MAP \cite{fonfw}. In \cite{kalton} it is proved that $\free{M}$ has the AP for a uniformly discrete metric space $M$, although the question whether it has the BAP is still open and of particular interest. 
	
	In his survey of Lipschitz free spaces \cite{godefroysurvey}, Godefroy asked whether the free space over any closed subset of a finite-dimensional Banach space has the MAP. It is known that $\free{M}$ has the $CN$-BAP for any $M\subseteq \R^N$, where $C$ is a universal constant, and $\R^N$ is equipped with an arbitrary norm \cite{apandschd}. This follows from the more general result that free spaces over doubling metric spaces have the BAP \cite{apandschd}*{Corollary 2.2}. In the case of the Euclidean norm on $\R^N$, we have the improvement that $\free{M}$ has the $C\sqrt{N}$-BAP, for an arbitrary subset $M\subseteq (\R^N, \nm\cdot _2)$ \cite{apandschd}*{Corollary 2.3}. In \cite{ps:15} it is shown that free spaces over compact subsets of $\R^N$ onto which there are certain Lipschitz `almost retractions' defined on slightly bigger subsets (see below for the precise formulation) have the MAP, where again $\R^N$ is equipped with an arbitrary norm. This includes the case of compact convex subsets of $\R^N$, which contrasts starkly with the infinite-dimensional case mentioned above. Also, it is known that $\free{M}$ has the MAP when $M$ is purely 1-unrectifiable, which is equivalent to the condition that $M$ contains no bi-Lipschitz image of a compact subset of $\R$ of positive measure. That $\free{M}$ has the MAP in this case follows because it is separable, has the BAP and is a dual space \cite{purely1unrect}*{Theorem B}. In \cite{apandschd} it is shown that $\free{\R^N, \nm\cdot _1}$ and $\free{l^1}$ have monotone finite-dimensional Schauder decompositions. In \cite{schauderbases} it is also proved that $\free{\R^N}$ and $\free{l^1}$ have Schauder bases, where $\R^N$ is equipped with an arbitrary norm.
	
	Motivated by Godefroy's question, in this paper we extend the class of subsets of $\R^N$ whose free spaces have the MAP by proving that the free space over a closed $C^1$-submanifold of $\R^N$ has the MAP, with respect to any norm. More precisely, we prove the following result.
	
	\begin{thm}\label{thm_main}
		Let $\nm\cdot$ be a norm on $\R^N$ and $M$ be a closed $C^1$-submanifold of $\R^N$ with an arbitrary distinguished point $x_0$. Equip $M$ with the metric $d(x,y) = \nm{x-y}$, $x,y\in M$. Then $\free{M}$ has the MAP.
	\end{thm}
	
	We repeat the main result of \cite{ps:15} as it is necessary to show that it cannot be applied to all $C^1$-submanifolds (even compact ones).
	
	\begin{thm}[\cite{ps:15}*{Theorem 1.1.}]\label{thm_ps}
		Let $N \geq 1$ and consider $\R^N$ equipped with some norm $\nm\cdot$. Let a compact set $M \subseteq \R^N$ have the property that given $\xi>0$, there exists a set $\hat{M}\subseteq \R^N$ and a Lipschitz map $\Psi\colon \hat{M}\to M$, such that $M \subseteq \inter{\hat{M}}$, $\Lip(\Psi)\leq 1 + \xi$ and $\nm{x-\Psi(x)} \leq \xi$ for all $x \in \hat{M}$. Then the Lipschitz-free space $\free{M}$ has the MAP.
	\end{thm}
	
	The rest of the paper is organised as follows. In Section \ref{sect_anc_supp}, we present a number of ancilliary results that support the proof of Theorem \ref{thm_main}. Section \ref{sect_proof} is devoted to the proof of Theorem \ref{thm_main} and in Section \ref{sect_open_problems} we make some remarks on open problems.
	
	We conclude the introduction by showing that there exists a $C^\infty$-norm on $\R^3$ whose unit sphere (which is a compact $C^\infty$-submanifold of $\R^3$) does not satisfy the assumptions of \Cref{thm_ps}. More precisely, we will show that there exist two norms $\nm\cdot, \nmt\cdot$ in $\R^3$, where $\nm\cdot$ is $C^\infty$, so that for the subset $M = S_{\nm\cdot}$ of $(\R^3,\nmt\cdot)$, the following condition does \emph{not} hold:
	\begin{enumerate}
		\item[($*$)]for every $\xi > 0$ there exists $\hat{M} \subset\R^3, M\subseteq \inter{\hat{M}}$ and a Lipschitz map $\Psi\colon \hat{M}\to M$ such that $\Lip(\Psi) \leq 1+\xi$ and $\nmt{\Psi(x)-x} \leq \xi$ for all $x\in\hat{M}$.
	\end{enumerate}
	
	It is known that if a Banach space $X$ is of dimension at least 3 then $X$ is isometrically isomorphic to a Hilbert space if and only if every subspace of $X$ is 1-complemented \cite{pb:87}*{9.3}. We can therefore fix a norm $\nmt\cdot$ on $\R^3$ such that the subspace $\R^2 \times \{0\}$ is not 1-complemented. To define $\nm\cdot$, let $\phi\colon [0,\infty)\to [0,\infty)$ be a convex $C^\infty$-function such that $\phi(1)=1$ and $\phi(t)=0$ for $t \in [0,\frac{3}{4}]$. Define the convex even $C^\infty$-function $\Phi\colon \R^3 \to [0,\infty)$ by
	\[
	\Phi(x) \;=\; \sum_{i=1}^3 \phi({\ts\frac{1}{2}}|x_i|),
	\]
	the convex set $B=\{x \in \R^N : \Phi(x) \leq 1\}$ and finally let $\nm\cdot$ be the Minkowski functional of $B$. An application of the Implicit Function Theorem demonstrates that $\nm\cdot$ is a $C^\infty$-norm; for more details see e.g.~\cite{dgz}*{Section V.1}.
	
	For a contradiction, assume that $M = S_{\nm\cdot}$ possesses the property labelled ($*$) above. It is easy to check that $\Phi(x)=1$ whenever $|x_1|,|x_2| \leq \frac{3}{2}$ and $x_3=2$, and thus all such points $x$ belong to $M$. Define $C = [0,1]^2\times\{0\}$ and $D=[-\frac{1}{2},\frac{3}{2}]\times\{0\}$. By translating $M$ we can see that for all sufficiently small $\xi > 0$ there exists an open set $U\subseteq\R^3$ that includes $C$ and a map $\Psi\colon U\to D$ such that $\Lip(\Psi) \leq 1+\xi$ and $\nmt{\Psi(x)-x}\leq\xi$ for every $x\in U$. By taking a convolution of $\Psi$ and a $C^\infty$-mollifier supported on a small enough neighbourhood of $0$, we can assume that $\Psi$ is $C^\infty$ and still satisfies the previous two inequalities. Choose some sufficiently small $\xi > 0$ and let $U$ and $\Psi$ be defined as previously. The total derivative $D\Psi(x,y,z)$ of $\Psi$ at $(x,y,z)$ maps $(\R^3,\nmt\cdot)$ linearly into $\R^2\times\{0\}$ and satisfies $\nm{D\Psi(x,y,z)}_{op} \leq 1+\xi$. Define the linear operator $T:\R^3 \to \R^2 \times \{0\}$ by $$T(a,b,c) = \int_{C} D\Psi(x,y,0)(a,b,c) \:\mathrm{d}x\mathrm{d}y.$$ We estimate
	\begin{align*}
	&\nmt{T(1,0,0) - (1,0,0)} \\=& \nmt{\int_0^1\int_0^1 D\Psi(x,y,0)(1,0,0) \:\mathrm{d}x\mathrm{d}y - (1,0,0)} \\=& \nmt{\int_0^1\int_0^1 \frac{d\Psi}{dx}(x,y,0) \:\mathrm{d}x\mathrm{d}y - (1,0,0)} \\=& \nmt{\int_0^1 \Psi(1,y,0) - \Psi(0,y,0) \:\mathrm{d}y - (1,0,0)} \\=& \nmt{\int_0^1 \Psi(1,y,0) - (1,y,0) - (\Psi(0,y,0)-(0,y,0)) \:\mathrm{d}y} \\ \leq& \int_0^1 \nmt{\Psi(1,y,0) - (1,y,0)}\:\mathrm{d}y + \int_0^1 \nmt{\Psi(0,y,0)-(0,y,0)} \:\mathrm{d}y \leq 2\xi.
	\end{align*}
	Similarly, we can show that $\nmt{T(0,1,0) - (0,1,0)} \leq 2\xi$. Now for all sufficiently large $n\in\N$, we choose $\Psi_n$ corresponding to $\xi_n = \frac{1}{n}$ and let $T_n$ be defined with respect to $\Psi_n$ as above. Since $\nm{T_n}_{op} \leq 1+\frac{1}{n}$, $(T_n)_n$ has a subsequence converging to some linear map $T$. Obviously $\nm T _{op} \leq 1$ and the range of $T$ is $\R^2\times\{0\}$. Also, by the above computations, the restriction of $T$ to $\R^2\times\{0\}$ is the identity. Therefore $T$ is a norm-one projection onto $\R^2\times\{0\}$, a contradiction.

	\section{Ancilliary results}\label{sect_anc_supp}
	
	Let $\R^N$ be equipped with the Euclidean norm $\nm\cdot _2$ and let $\nm\cdot$ be some other norm. Fix $K \geq 1$ such that $$K^{-1}\nm\cdot _2 \;\leq\; \nm\cdot \;\leq\; K \nm\cdot _2.$$ Denote the closed Euclidean ball of radius $R>0$ centered at $y\in\R^N$ by $B_R(y)$ and write $B_R = B_R(0)$. For a subset $U\subseteq\R^N$, write $\overline{U}$ for the closure of $U$. For a Lipschitz function $f\colon U\to\R$, where $U\subseteq\R^N$ is some set, write $\Lip_{\nm\cdot}(f)$ and $\Lip_{\nm\cdot _2}(f)$ for its Lipschitz constants measured with respect to $\nm\cdot$ and $\nm\cdot _2$, respectively. By $\Lip(f)$ we will mean $\Lip_{\nm\cdot}(f)$. Define also $\nm f _\infty = \sup\{|f(x)| : x\in U\}$. Given a differentiable map $\phi\colon U\to\R^{N'}$, where $U\subseteq \R^N$ is open, we denote by $D\phi(x)$ the total derivative of $\phi$ at $x$, which is a linear operator from $\R^N$ to $\R^{N'}$. The following two definitions are the same as in \cite{leo-stein}*{Definition 2}.
	
	\begin{defn}
		Let $d,k\in\N$, $1\leq d\leq N$. The subset $M\subseteq\R^N$ is called a \emph{$d$-dimensional $C^k$-submanifold} of $\R^N$ if for every $x\in M$ there exist open sets $V,U\subseteq\R^N$ and a $C^k$-diffeomorphism $\phi\colon V \to U$ such that $0\in V, x\in U$, $\phi(0) = x$ and $$\phi((\R^d \times \{(\underbrace{0,\ldots,0}_{N-d})\}) \cap V) = M\cap U.$$
	\end{defn}
	
	\begin{defn}
		If $M$ is a $d$-dimensional $C^k$-submanifold of $\R^N$ and $x\in M$, we define the \emph{tangent space} at $x$: $$T_x=\{v\in\R^N:  \exists \gamma\colon (-1,1)\to M \text{ a $C^1$-curve such that } \gamma(0) = x \text{ and }\gamma'(0)=v\}.$$
	\end{defn}
	Note that the tangent space $T_x$ is a $d$-dimensional vector subspace of $\R^N$. Denote the orthogonal projection onto $T_x$ by $P_x$.
	
	Let $M$ be a closed $d$-dimensional $C^1$-submanifold of $\R^N$. By translation, if necessary, we can assume that $0\in M$. Let $\Lipo{M}$ be the Banach space of all Lipschitz functions on $M$ vanishing at $0$ equipped with the norm $\Lip_{\nm\cdot}(\cdot)$, and denote the closed unit ball of $\Lipo{M}$ by $B_{\Lipo{M}}$. 
	
	We begin with some basic lemmas and propositions that allow us to mollify a given Lipschitz function on bounded subsets of $M$ without increasing its Lipschitz constant by much.
	
	\begin{lem}\label{lemma:1}
		Let $R,\epsilon > 0$. Then there exists $\delta > 0$ such that $\nm{y-x-P_x(y-x)}_2 \leq \epsilon\nm{y-x}_2$ whenever $x,y\in M\cap B_R$ satisfy $\nm{x-y}_2\leq\delta$.
	\end{lem}
	\begin{proof}
		By compactness of $M\cap B_R$, it suffices to prove the lemma locally, that is, that for every $q\in M$ there exists a neighbourhood $U$ of $q$ in $M$ such that for every $\epsilon>0$ there exists $\delta>0$ such that $$\nm{y-x - P_x(y-x)}_2 \leq \epsilon \nm{y-x}_2$$ for $x,y\in U, \nm{x-y}_2 \leq \delta$.
		
		Let $q\in M$ and let $\phi$ be a $C^1$-submanifold chart around $q$, i.e.~ $\phi\colon V \to U$ is a $C^1$ diffeomorphism between the open sets $V,U\subseteq\R^N$ such that $0\in V, q\in U$, $\phi(0) = q$, $\phi(\R^d \cap V) = M\cap U$, and $\R^d \cong \R^d \times \{0\} \subseteq \R^N$. It is not hard to see that the $D\phi(v)$ maps $\R^d$ to $T_{\phi(v)}$ for every $v\in V$. 
		
		Let $\epsilon > 0$. Using the fact that $\phi$ is locally bi-Lipschitz and $D\phi$ is continuous, we shrink $V$ and $U$, if necessary, so that $V$ is convex and there exists $a\geq 1$ such that $$\frac{1}{a}\nm{v-w}_2 \leq \nm{\phi(v) - \phi(w)}_2 \leq a\nm{v-w}_2 \quad\text{and}\quad \nm{D\phi(v) - D\phi(w)}_2 \leq \frac{\epsilon}{a}$$ for all $v,w\in V$.
		
		Choose $x,y\in M\cap U$ and let $v=\phi^{-1}(x)$ and $w=\phi^{-1}(y)$. Define $f\colon [0,1] \to M$ by $f(t) = \phi(v+t(w-v))$. Then $$y-x = f(1)-f(0) = \lint{0}{1}{f'(t)}{t} = \lint{0}{1}{D\phi(v+t(w-v))(w-v)}{t},$$ meaning 
		\begin{align*}
		\nm{y-x - D\phi(v)(w-v)}_2 &\leq \lint{0}{1}{\nm{(D\phi(v+t(w-v))-D\phi(v))(w-v)}_2}{t} \\&\leq \frac{\epsilon}{a} \nm{w-v}_2 \leq \epsilon\nm{y-x}_2.
		\end{align*}
		Since $D\phi(v)(w-v) \in T_{\phi(v)} = T_x$ and $P_x (y-x)$ is the closest point to $y-x$ in $T_x$, we have \[\nm{y-x - P_x(y-x)}_2 \leq \nm{y-x - D\phi(v)(w-v)}_2 \leq \epsilon\nm{y-x}_2. \qedhere\]
	\end{proof}
	
	According to \cite{whitney}*{Lemma 23}, there exists an open neighbourhood $\ce{M}$ of $M$ and a $C^1$ map $\psi\colon \ce{M}\to M$ such that $\psi(x) = x$ for $x\in M$ (in \cite{whitney} $\ce{M}$ is called $R(M)$ and $\psi$ is called $H$). Given $\epsilon>0$, let $M^\epsilon = \{y\in\R^N : d(y,M) < \epsilon\}$, where $d(y,M) = \inf_{x\in M} \nm{y-x}_2$. We will fix the map $\psi$ for the rest of the paper. Next is a straightforward but important observation about $\psi$ for use later on.
	
	\begin{lem}\label{lm:dpsiproj}
		For $x\in M$ we have $(D\psi(x)-I)|_{T_x} = 0$.
	\end{lem}
	\begin{proof}
		Let $x\in M$ and $v\in T_x$. We need to show that $D\psi(x)(v) = v$. Let $\phi\colon V \to U$ be a $C^1$-submanifold chart around $x$, where $V,U\subseteq\R^N$, $0\in V, x\in U$, $\phi(0) = x$. As $D\phi(0)$ maps $\R^d$ onto $T_x$, there exists $u\in \R^d$ such that $D\phi(0)(u)=v$. Since $(\psi\circ\phi)|_{\R^d} = \phi|_{\R^d}$, we indeed have \[D\psi(x)(v) = D(\psi\circ\phi)(0)(u) = D((\psi\circ\phi)|_{\R^d})(0)(u) = D(\phi|_{\R^d})(0)(u) = v. \qedhere\]
	\end{proof}
	
	The following lemma will be of key importance when we mollify Lipschitz funcitons defined on bounded subsets of $M$.
	\begin{lem}\label{lemma:2}
		For every $R,\epsilon>0$ there exists $\delta > 0$ such that $$\nm{\psi(x+z)-\psi(y+z)-(x-y)}_2 \leq \epsilon \nm{x-y}_2$$ for all $x,y\in M\cap B_R$ and $z\in\R^N$ such that $\nm{x-y}_2, \nm z _2\leq \delta$.
	\end{lem}
	\begin{proof}
		Let $R,\epsilon>0$. By uniform continuity of $D\psi$ on compact sets, we choose $\delta_1 > 0$ such that $\overline{M^{2\delta_1}} \cap B_{R+2\delta_1} \subseteq \ce{M}$ and
		\begin{align}
		\nm{D\psi(q)-D\psi(u)}_2 \leq \frac{\epsilon}{2} \text{\; whenever \;} q,u\in \overline{M^{2\delta_1}} \cap B_R, \nm{q-u}_2\leq 2\delta_1. \label{eq:Dpsi}
		\end{align} 
		Set $$B = \max\{\nm{D\psi(q)-I}_2 : q\in \overline{M^{2\delta_1}} \cap B_{R+2\delta_1}\}.$$ We apply \Cref{lemma:1} to $R$ and $\frac{\epsilon}{2B}$ to find a corresponding $\delta_2 > 0$. Set $\delta = \min(\delta_1, \delta_2)$ and let $x,y\in M\cap B_R, \nm{x-y}_2, \nm z _2\leq \delta$.
		
		For $t\in[0,1]$ we have $\nm{x+z+t(y-x) - x}_2 \leq 2\delta_1$ and so $x+z+t(y-x) \in \overline{M^{2\delta_1}} \cap B_{R+2\delta_1}$. Therefore we can define $f(t) = \psi(x+z+t(y-x)) - (x+ t(y-x))$, $t\in [0,1]$. We have
		\begin{align*}
		\nm{\psi(x+z)-\psi(y+z)-(x-y)}_2 &= \nm{f(1) - f(0)}_2 = \nm{\lint{0}{1}{f'(t)}{t}}_2 \\&\leq \lint{0}{1}{\nm{f'(t)}_2}{t}.
		\end{align*}
		Furthermore,
		\begin{align*}
		&\nm{f'(t)}_2 \\
		=& \nm{D\psi(x+z+t(y-x))(y-x) - (y-x)}_2 \\
		\leq& \nm{(D\psi(x+z+t(y-x)) - D\psi(x))(y-x)}_2 + \nm{(D\psi(x)-I)(y-x)}_2 \\
		\leq& \frac{\epsilon}{2}\nm{y-x}_2 + \nm{(D\psi(x)-I)(y-x - P_x(y-x))}_2 \tag*{by \eqref{eq:Dpsi} and \Cref{lm:dpsiproj}} \\
		\leq& \frac{\epsilon}{2}\nm{y-x}_2 + B\frac{\epsilon}{2B}\nm{y-x}_2 = \epsilon\nm{y-x}_2. \qedhere
		\end{align*}
	\end{proof}
	
	We next define a certain convolution of a given Lipschitz function $f$ on $M$ with a standard mollifier. First, define $\nu\colon \R^N \to [0,+\infty)$ by
	\begin{equation*}
	\nu(x) = 
	\begin{cases}
	A\exp \left(\frac{1}{\nm x ^2 _2 - 1}\right) & \text{ if } \nm x _2 < 1, \\
	0 & \text{ if } \nm x _2 \geq 1,
	\end{cases}
	\end{equation*}
	where $A>0$ is chosen so that $\lint{\R^N}{}{\nu (x)}{x} = 1$. Next, for each $s>0$ we put $$\nu_s(x) = \frac{1}{s^N}\nu\left(\frac{x}{s}\right).$$ The function $\nu_s$ is $C^\infty$ and satisfies $\lint{\R^N}{}{\nu_s (x)}{x} = 1$ and $\supp(\nu_s) \subseteq B_s(x)$.
	
	Let $R\geq 1$ be fixed. Choose $\delta_0 > 0$ such that $\overline{M^{2\delta_0}} \cap B_{R+2\delta_0} \subseteq \ce{M}$. Fix $s\in (0,\delta_0]$ and let $f\in \Lipo{M}$. For $x \in \overline{M^{\delta_0}} \cap B_{R+\delta_0}$ we define $$\hat{f}_s(x) = \lint{B_s}{}{\nu_s (z)f(\psi(x+z))}{z} = \lint{x+B_s}{}{\nu_s (z-x)f(\psi(z))}{z}.$$
	
	In the next two results, we fix $f\in B_{\Lipo{M}}$ and estimate the Lipschitz constant of $\hat f _s$ by first considering points that are close together and then far from each other. Since $\psi$ is $C^1$ on $\ce{M}$, it is Lipschitz on the compact set $\overline{M^{2\delta_0}} \cap B_{R+2\delta_0}$. Observe that the next result holds even if, when restricted to this set, the Lipschitz constant of $\psi$ (relative to $\nm\cdot$) is large.
	
	\begin{prop}\label{prop:xyclose}
		Let $\epsilon > 0$. There exists $\delta \in (0,\delta_0]$ such that $$|\hat f _s(x) - \hat f _s(y)| \leq (1+\epsilon)\nm{x-y}$$ whenever $x,y\in M\cap B_{R+\delta_0}$, $\nm{x-y}\leq\delta$ and $s \in (0,\delta]$.
	\end{prop}
	\begin{proof}
		Let $\epsilon > 0$. According to \Cref{lemma:2} we can choose $\delta \in (0,\delta_0]$ such that $$\nm{\psi(x+z) - \psi(y+z) - (x-y)}_2 \leq \frac{\epsilon}{K^2} \nm{x-y}_2$$ for all $x,y\in M\cap B_{R+\delta_0}, \nm{x-y}_2 \leq K \delta$ and $z\in\R^N, \nm z _2 \leq \delta$. Let $s \in (0,\delta]$ and pick $x,y\in M\cap B_{R+\delta_0}, \nm{x-y}\leq\delta$. We have $\nm{x-y}_2 \leq K\delta$ and if $\nm z _2 \leq s$ then
		\begin{align*}
		\nm{\psi(x+z) - \psi(y+z) - (x-y)} &\leq K\nm{\psi(x+z) - \psi(y+z) - (x-y)}_2 \\&\leq \frac{\epsilon}{K} \nm{x-y}_2 \leq \epsilon \nm{x-y}.
		\end{align*}
		Hence $\nm{\psi(x+z) - \psi(y+z)} \leq (1+\epsilon)\nm{x-y}$.
		
		We conclude that 
		\begin{align*}
		|\hat f _s(x) - \hat f _s(y)| &\leq \lint{B_s}{}{\nu_s(z) |f(\psi(x+z)) - f(\psi(y+z))|}{z} \\&\leq \lint{B_s}{}{\nu_s(z) \nm{\psi(x+z) - \psi(y+z)}}{z} \leq (1+\epsilon)\nm{x-y}. \tag*{\qedhere}
		\end{align*}
	\end{proof}
	
	Next, we make an estimate in the case where points are far apart. Let $L \geq 1$ be the Lipschitz constant of the restriction of $\psi$ to $\overline{M^{2\delta_0}} \cap B_{R+2\delta_0}$, with respect to $\nm\cdot _2$.
	
	\begin{prop}\label{prop:xyfar}
		Let $\epsilon,\delta > 0$. Then $$|\hat{f}_s(x) - \hat{f}_s(y)| \leq (1+\epsilon) \nm{x-y}$$ whenever $x,y\in M\cap B_{R+\delta_0}, \nm{x-y}\geq\delta$ and $s \in \left(0,\min\left( \frac{\epsilon\delta}{2LK}, \delta_0 \right)\right]$.
	\end{prop}
	\begin{proof}
		Let $x,y\in M\cap B_{R+\delta_0}$ be such that $\nm{x-y}\geq\delta$, and let $s \in \left(0,\min\left( \frac{\epsilon\delta}{2LK}, \delta_0 \right)\right]$. We have
		\begin{align*}
		|\hat{f}_s(x) - f(x)| &\leq \lint{B_s}{}{\nu_s(z) |f(\psi(x+z)) - f(x)|}{z} \\&\leq \lint{B_s}{}{\nu_s(z) \nm{\psi(x+z) - x}}{z} \leq LKs,
		\end{align*}
		and similarly for $y$. Therefore
		\begin{align*}
		|\hat f _s (x) - \hat f _s(y)| &\leq |\hat f _s(x) - f(x)| + |f(x) - f(y)| + |f(y) - \hat f _s(y)| \\&\leq 2LKs + \nm{x-y} \leq \epsilon\delta + \nm{x-y} \leq (1+\epsilon)\nm{x-y}. \qedhere
		\end{align*}
	\end{proof}
	
	We require some estimates concerning $D\hat f _s$. In the next two lemmas, $\nm{D\nu_s (z)}_2$ and $\nm{D\hat f _s (x)}_2$ denote the norms of the operators $D\nu_s (z)$ and $D\hat f _s (x)$ with respect to $\nm\cdot _2$, respectively.
	\begin{lem}
		There exists $G>0$ (depending only on $N$) such that $$\lint{B_s}{}{\nm{D\nu_s(z)}_2}{z} \leq \frac{G}{s}$$ for all $s>0$.
	\end{lem}
	\begin{proof}
		Set $G = (\mathrm{e}\lint{0}{1}{\exp \left(\frac{1}{r ^2 - 1}\right) r^{N-1}}{r})^{-1}$ and let $\Gamma$ denote the area of the Euclidean unit sphere $S^{N-1} \subseteq \R^N$. Given $z \in \R^N$, $\nm z _2 < 1$, we have
		\[
		D\nu(z) = -A\exp\left(\frac{1}{\nm z _2^2 -1}\right) \frac{2\ip{z}{\cdot}}{\left(\nm z _2^2 -1\right)^2},
		\]
		where $\ip{\cdot}{\cdot}$ denotes the inner product in $\R^N$. Let $s>0$. As $\nu_s(z) = \frac{1}{s^N}\nu\left(\frac{z}{s}\right)$, we obtain $D\nu_s (z) = \frac{1}{s^{N+1}} D\nu\left(\frac{z}{s}\right)$ and thus $$\nm{D\nu_s (z)}_2 = \frac{A}{s^{N+1}} \exp \left(\frac{1}{\nm{\frac{z}{s}}_2^2 -1}\right)\frac{2\nm{\frac{z}{s}}_2}{\left(\nm {\frac{z}{s}}_2^2 -1\right)^2}.$$ This function of $z$ is radially symmetric, hence
		\begin{align*}
		\lint{B_s}{}{\nm{D\nu_s (z)}_2}{z} &= \frac{A \Gamma}{s^{N+1}} \lint{0}{s}{\exp \left(\frac{1}{\left(\frac{r}{s}\right)^2  -1}\right)\frac{2\left(\frac{r}{s}\right)}{\left(\left(\frac{r}{s}\right)^2 -1\right)^2} r^{N-1}}{r} \\&\leq \frac{A \Gamma}{s} \lint{0}{1}{\exp \left(\frac{1}{u^2 -1}\right)\frac{2u}{\left(u^2 -1\right)^2}}{u} = \frac{A \Gamma}{\mathrm{e}s}.
		\end{align*}
		The result now follows from the fact that \begin{equation}\frac{1}{A} = \lint{B_1}{}{\exp \left(\frac{1}{\nm z ^2 _2 - 1}\right)}{z} = \frac{\Gamma}{\mathrm{e}G}. \tag*{\qedhere}\end{equation}
	\end{proof}
	
	\begin{lem} \label{lem:fhatiscinf}
		For $f\in\Lipo{M}$, the function $\hat f _s$ is $C^\infty$ on $M^{\delta_0} \cap \inter{B_{R+\delta_0}}$ and satisfies $$\nm{D\hat f _s (x)}_2 \leq KL\Lip(f),$$ and $$\nm{D\hat f _s (x) - D\hat f _s (y)}_2 \leq \frac{GKL}{s}\Lip(f)\nm{x-y}_2,$$ for every $x,y\in M^{\delta_0} \cap \inter{B_{R+\delta_0}}$.
	\end{lem}
	\begin{proof}
		By \cite{evans10}*{Appendix C.4, Theorem 6(i)} we have that $\hat f _s$ is $C^\infty$ on $M^{\delta_0} \cap \inter{B_{R+\delta_0}}$ and $$D\hat f _s (x) = \lint{B_s}{}{D\nu_s(z) f(\psi(x+z))}{z}.$$ If $x,y\in M^{\delta_0} \cap \inter{B_{R+\delta_0}}$ then $$|\hat f _s (x) - \hat f _s(y)| \leq \lint{B_s}{}{\nu_s(z) |f(\psi(x+z)) - f(\psi(y+z))|}{z} \leq KL\Lip(f)\nm{x-y}_2.$$ Therefore $$\nm{D\hat f _s(x)}_2 \leq \Lip_{\nm\cdot _2} (\hat f_s) \leq KL\Lip(f).$$ Also,
		\begin{align*}
		&\nm{D\hat f _s (x) - D\hat f _s(y)}_2 \leq \lint{B_s}{}{\nm{D\nu_s(z)}_2 |f(\psi(x+z)) - f(\psi(y+z))|}{z} \\&\leq KL\Lip(f) \lint{B_s}{}{\nm{D\nu_s(z)}_2 }{z}\nm{x-y}_2 \leq \frac{GKL}{s}\Lip(f)\nm{x-y}_2. \qedhere
		\end{align*}
	\end{proof}
	
	In order to obtain suitable finite-rank operators that witness the MAP, we make use of an interpolation process first used in \cite{apandschd} and again in \cite{ps:15}. 
	
	Fix $w\in \R^d$. We deﬁne a closed hypercube $C\subseteq \R^d$ having edge length $\delta>0$ and vertices $v_\gamma \in\R^d, \gamma\in \{0,1\}^d$, given by $v_\gamma = w + \delta\gamma$.
	Let $f$ be a real-valued function whose domain of deﬁnition includes the set of vertices of $C$. We deﬁne the `interpolation function' $\Lambda(f,C)$ on $\R^d$ by $$\Lambda(f,C)(x) = \sum_{\gamma\in\{0,1\}^{d}} \left( \prod_{i=1}^d \left(1-\gamma_i+(-1)^{\gamma_i + 1} \frac{x_i-w_i}{\delta}\right)\right) f(v_\gamma).$$ Note that $\Lambda(f,C)$ is the unique function that agrees with $f$ on the vertices of $C$ and is coordinatewise affine, i.e.~$t\mapsto \Lambda(f,C)(x_1,\ldots,x_{i-1},t,x_{i+1},\ldots,x_d)$ is affine whenever $1\leq i \leq d$.
	
	The following lemma is very similar to \cite{ps:15}*{Lemma 3.3}, with the important difference that the estimate $$\Lip_{\nm\cdot _2}((\Lambda(f,C) - f)|_C) \leq (1+\sqrt{d})\epsilon.$$ is in place of \cite{ps:15}*{Lemma 3.3 (14)}. Despite the similarity between the two results, we provide a proof for completeness and ask the reader to excuse any redundancy. 
	
	\begin{lem} \label{lm:grid}
		Let $f$ be a $C^1$ function defined on a convex neighbourhood $V$ of a given hypercube $C\subseteq (\R^{d}, \nm\cdot _2)$ as above. Let $\epsilon > 0$ and suppose that there is $\delta > 0$ such that $\nm{Df(u) - D f(u')}_2 \leq \epsilon$ if $\nm{u-u'}_2\leq \sqrt d\delta$. If $C$ has sidelength $l\leq \delta$ then $$\Lip_{\nm\cdot _2}((\Lambda(f,C) - f)|_C) \leq (1+\sqrt d)\epsilon,$$ and $$\nm{(\Lambda(f,C) - f)|_C}_\infty \leq \sqrt d\delta \Lip_{\nm\cdot _2}(f).$$ 
	\end{lem}
	\begin{proof}
		First we show that for $u\in V, h\in\R^d$ such that $u+h\in V$ and $\nm h _2 \leq \sqrt d\delta$ we have 
		\begin{align}\label{eqeq}
		|f(u+h) - f(u) - Df(u)(h)| \leq \epsilon \nm h _2. 
		\end{align} 
		For such $u$ and $h$, using the Mean Value Theorem we can find a vector $u'$ on the line segment between $u$ and $u+h$ such that $f(u+h) - f(u) = Df(u')(h)$. Then $\nm{u-u'}_2 \leq \sqrt d\delta$ and 
		\begin{align*}
		|f(u+h) - f(u) - Df(u)(h)| &= |Df(u')(h) - Df(u)(h)| \\&\leq \nm{Df(u') - Df(u)}_2 \nm h _2 \leq \epsilon \nm h _2.
		\end{align*}
		
		Now let $z\in \inter{C}$. The first equality on \cite{ps:15}*{p.41} is 
		\begin{align}\label{eq:lincomb}
		D\Lambda(f,C)(z) = \sum_{\gamma\in\{0,1\}^{d}} c_\gamma \left(\frac{f(v_\gamma + (-1)^{\gamma_j}l e_j) - f(v_\gamma)}{(-1)^{\gamma_j} l} \right)_{j=1}^d,
		\end{align} where $$c_\gamma =  \prod_{i=1}^d \left(1-\gamma_i+(-1)^{\gamma_i + 1} \frac{z_i-w_i}{l}\right) \geq 0, \sum_{\gamma\in\{0,1\}^{d}} c_\gamma = 1.$$ 
		
		For every $\gamma\in \{0,1\}^d$ we have
		\begin{align*}
		&\nm{ \left(\frac{f(v_\gamma + (-1)^{\gamma_j}l e_j) - f(v_\gamma)}{(-1)^{\gamma_j} l} \right)_{j=1}^d - Df(v_\gamma) }_2 \\=& \nm{ \left(\frac{f(v_\gamma + (-1)^{\gamma_j}l e_j) - f(v_\gamma) - Df(v_\gamma)((-1)^{\gamma_j} l e_j)}{(-1)^{\gamma_j} l} \right)_{j=1}^d }_2 \\  \leq& \frac{\sqrt d}{l} \nm{ ( f(v_\gamma + (-1)^{\gamma_j}l e_j) - f(v_\gamma) - Df(v_\gamma)((-1)^{\gamma_j} l e_j))_{j=1}^d }_\infty \\ \leq& \frac{\sqrt d \epsilon}{l} \nm{(-1)^{\gamma_j}l e_j}_2 = \sqrt d \epsilon \tag*{by \eqref{eqeq}}.
		\end{align*}
		Therefore by \eqref{eq:lincomb},
		\begin{align*}
		&\nm{D\Lambda(f,C)(z) - Df(z)}_2 \\ \leq& \nm{D\Lambda(f,C)(z) - \sum_{\gamma\in\{0,1\}} c_\gamma Df(v_\gamma) }_2 + \nm{\sum_{\gamma\in\{0,1\}} c_\gamma Df(v_\gamma) - Df(z)}_2 \\ \leq& \sqrt d \epsilon + \epsilon = (1+\sqrt d)\epsilon
		\end{align*}
		since on the second line we have $\nm{v_\gamma-z}_2 \leq \sqrt d l \leq \sqrt d \delta$. Therefore $$\Lip_{\nm\cdot _2}((\Lambda(f,C)-f)|_{\inter{C}}) \leq (1+\sqrt d)\epsilon,$$ and due to continuity this extends to $C$, as required.
		
		For the uniform norm, we have 
		\begin{align*}
		|\Lambda(f,C)(z) - f(z)| &= \left| \sum_{\gamma\in\{0,1\}^{d}} c_\gamma f(v_\gamma) - f(z) \right| \leq \sum_{\gamma\in\{0,1\}^{d}} c_\gamma |f(v_\gamma) - f(z)| \\&\leq \sum_{\gamma\in\{0,1\}^{d}} c_\gamma \Lip_{\nm\cdot _2}(f)\nm{v_\gamma - z}_2 \leq \sqrt d\delta \Lip_{\nm\cdot _2}(f). \tag*{\qedhere}
		\end{align*}
	\end{proof}
	
	Next we prove a lemma about gluing functions together using a partition of unity. This lemma is again similar to \cite{ps:15}*{Lemma 2.7}.
	
	\begin{lem} \label{lm:gluing}
		Let $U$ be a subset of $\R^N$ and $f\colon U\to\R$ be a Lipschitz function. Suppose that $\epsilon>0$, $(U_i)_{i=1}^m$ is an open cover of $U$, $(\alpha_i)_{i=1}^m$ a partition of unity subordinate to $(U_i)_{i=1}^m$ consisting of $H$-Lipschitz functions, and $f_i\colon U_i\to \R, 1\leq i\leq m$ satisfy $\nm{f_i-f|_{U_i}}_\infty, \Lip_{\nm\cdot}(f_i-f|_{U_i}) < \epsilon$. If $g = \sum_{i=1}^m \alpha_i f_i$ then $\nm{g-f}_\infty \leq \epsilon$ and $\Lip_{\nm\cdot}(g-f) \leq (1+mH)\epsilon$.
	\end{lem}
	\begin{proof}
		For any $x\in U$ we have $$|g(x)-f(x)| = \left|\sum_{i=1}^m \alpha_i(x) (f_i(x) - f(x))\right| \leq \epsilon \sum_{i=1}^m \alpha_i(x) = \epsilon,$$ and therefore $\nm{g-f}_\infty \leq \epsilon$. Now pick different $x,y\in U$. Let $A = \{i : x\in U_i\}$ and $B = \{i : y\in U_i\}$. We have 
		\begin{align*}
		|g(x)-f(x) - (g(y)-f(y))| =& \left|\sum_{i\in A} \alpha_i(x) (f_i(x)-f(x)) - \sum_{i\in B} \alpha_i(y) (f_i(y)-f(y))\right| \\
		\leq& \sum_{i\in A\cap B} |\alpha_i(x) (f_i(x)-f(x)) - \alpha_i(y) (f_i(y)-f(y))| \\
		+& \sum_{i\in A\setminus B} \alpha_i(x) |f_i(x)-f(x)| + \sum_{i\in B\setminus A} \alpha_i(y) |f_i(y)-f(y)| \\
		\leq& \sum_{i\in A\cap B} \alpha_i(x) |f_i(x)-f(x) - (f_i(y)-f(y))| \\
		+& \sum_{i\in A\cap B} |(\alpha_i(x)-\alpha_i(y))(f_i(y)-f(y))| \\
		+& \epsilon \left(\sum_{i\in A\setminus B} \alpha_i(x) + \sum_{i\in B\setminus A} \alpha_i(y)\right) \\
		\leq& \epsilon\nm{x-y} + \card(A\cap B)\epsilon H\nm{x-y} \\
		+& \epsilon \left(\sum_{i\in A\setminus B} |\alpha_i(x)-\alpha_i(y)| +  \sum_{i\in B\setminus A} |\alpha_i(y)-\alpha_i(x)|\right) \\
		\leq& (1+mH)\epsilon\nm{x-y},
		\end{align*}
		where $\card(A\cap B)$ is the cardinality of $A\cap B$. Therefore $\Lip_{\nm\cdot}(g-f) \leq (1+mH)\epsilon$.
	\end{proof}
	
	Finally, we define a `flattening' operator $\Phi$, following an idea taken from \cite{godefroykalton}*{Proposition 5.1}. Let $R>0$ be fixed.
	
	Let $\mu\colon [0, \infty) \to [0,1]$ be defined by
	\begin{equation}
	\mu(t) =
	\begin{cases}
	1, & \text{if } t \leq R \\
	(\log{R})^{-1} (2\log{R} - \log{t}), & \text{if } R\leq t\leq R^2 \\
	0 & \text{otherwise. } \\
	\end{cases}
	\end{equation}
	
	Denote by $\Lipo{M\cap B_{R^2}}$ the space of all Lipschitz functions on $M\cap B_{R^2}$ vanishing at $0$, equipped with the norm $\Lip(\cdot)$. Given $f\in\Lipo{M\cap B_{R^2}}$ and $x\in M$, we define
	\begin{equation*}
	\Phi(f)(x) =
	\begin{cases}
	\mu(\nm{x}_2)f(x), & \text{if } x\in B_{R^2} \\
	0 & \text{otherwise. } \\
	\end{cases}
	\end{equation*}
	
	It is clear that $\Phi(f)$ agrees with $f$ on $M\cap B_R$ and vanishes outside $B_{R^2}$. In particular, $\Phi(f)(0) = 0$.
	\begin{prop} \label{prop:flat}
		The map $\Phi\colon \Lipo{M\cap B_{R^2}} \to \Lipo{M}$ is a bounded operator satisfying $$\nm{\Phi} \leq 1+\frac{K^2}{\log{R}}.$$
	\end{prop}
	\begin{proof}
		Let $f\in B_{\Lipo{M\cap B_{R^2}}}$. Extend $f$ to $\R^N$ while preserving its Lipschitz constant, by McShane's Extension Theorem \cite{weaver}*{Theorem 1.33} and define $g(x)=\mu(\nm x _2)f(x), x\in \R^N$. It suffices to prove that $\Lip(g) \leq 1+\frac{K^2}{\log{R}}$.  By Rademacher's Theorem, it is enough to prove $\nm{Dg(x)} \leq 1+\frac{K^2}{\log{R}}$ whenever $x$ is a point of differentiability of both $\mu \circ \nm\cdot _2$ and $f$. Given such $x$, if $\nm x _2 < R$ or $\nm x _2 > R^2$ then $Dg(x) = Df(x)$ or $Dg(x) = 0$, respectively. If $R\leq \nm x _2 \leq R^2$ then 
		\begin{align*}
		\nm{Dg(x)} &= \nm{\mu'(\nm x _2)\frac{\ip{x}{\cdot}}{\nm x _2} f(x) + \mu(\nm x _2)Df(x)} \\&\leq \nm{\mu'(\nm x _2)\frac{\ip{x}{\cdot}}{\nm x _2} f(x)} + \nm{\mu(\nm x _2)Df(x)} \\&\leq 1 + K\nm{\mu'(\nm x _2)\frac{\ip{x}{\cdot}}{\nm x _2} f(x)}_2 = 1 + \frac{K|f(x)|}{\nm x _2 \log R} \leq 1 + \frac{K^2}{\log R}. \qedhere
		\end{align*}
	\end{proof}
	
	\section{Proof of \Cref{thm_main}}\label{sect_proof}
	
	We will build a sequence of operators in the proof of \Cref{thm_main}. Before giving the proof, we present three results that supply the components needed to assemble these operators.
	
	First we construct `mollifier' operators $S_n$ defined on $\Lipo{M}$. Fix an $n\in\N$ and a $\delta_n>0$ such that $\overline{M^{2\delta_n}} \cap B_{n^2+2\delta_n} \subseteq \mathcal{E}(M).$ Let $L_n \geq 1$ be the Lipschitz constant of the restriction of $\psi$ to $\overline{M^{2\delta_n}} \cap B_{n^2+2\delta_n}$, with respect to $\nm{\cdot}_2$. Recall the functions $\hat{f}_s$ defined before \Cref{prop:xyclose}.
	
	\begin{thm} \label{thm:smoothingext}
		There exists $s_n\in (0,\delta_n]$ such that the linear map $S_n\colon\Lipo{M}\to\Lipo{M\cap B_{n^2+\delta_n}}$, $$S_n(f)(x) = \hat{f} _{s_n}(x) - \hat{f} _{s_n}(0),$$ satisfies $\nm{S_n} \leq 1+\frac{1}{n}$ and $\nm{S_n(f)-f|_{M\cap B_{n^2+\delta_n}}}_\infty \leq \frac{1}{n}$ for all $f\in B_{\Lipo{M}}$. Moreover, if $(f^{(k)})_k$ is a bounded sequence of functions in $\Lipo{M}$ converging pointwise to $f\in\Lipo{M}$ then $S_n(f_k) \to S_n(f)$ pointwise as $k\to\infty$.
	\end{thm}
	\begin{proof}
		Clearly $S_n(f)(0) = 0$. By \Cref{prop:xyclose} there exists $\delta \in (0,\delta_n]$ such that $|\hat{f}_{s}(x)-\hat{f}_{s}(y)| \leq (1+\frac{1}{n})\nm{x-y}$ whenever $f\in B_{\Lipo{M}}$, $x,y\in M\cap B_{n^2+\delta_n}$, $\nm{x-y} \leq \delta$ and $s \in (0,\delta]$. Set $s_n = \frac{\delta}{2nKL_n} < \delta \leq \delta_n$. By \Cref{prop:xyfar} we have $|\hat{f} _{s_n}(x)-\hat{f} _{s_n}(y)| \leq (1+\frac{1}{n})\nm{x-y}$ whenever $x,y\in M\cap B_{n^2+\delta_n}$ and $\nm{x-y}\geq \delta$.  Therefore $|\hat{f} _{s_n}(x)-\hat{f} _{s_n}(y)| \leq (1+\frac{1}{n})\nm{x-y}$ for all $x,y\in M\cap B_{n^2+\delta_n}$. It follows that $\Lip(S_n(f)) \leq 1+\frac{1}{n}$ and hence $\nm{S_n} \leq 1+\frac{1}{n}$.
		
		If $x\in M\cap B_{n^2+\delta_n}$ and $\nm z _2 \leq s_n$ then $$\nm{x-\psi(x+z)} \leq K\nm{\psi(x)-\psi(x+z)}_2 \leq KL_n\nm{z}_2 \leq KL_n s_n.$$ Therefore if $\Lip(f)\leq 1$ then $$|f(x)-\hat{f} _{s_n}(x)| \leq \lint{B_{s_n}}{}{\nu_{s_n}(z)|f(x)-f(\psi(x+z))|}{z} \leq KL_n s_n \leq \frac{1}{2n}.$$ Hence $$|S_n(f)(x)-f(x)| \leq |\hat{f} _{s_n}(x) - f(x)| + |\hat{f} _{s_n}(x_0) - f(x_0)| \leq \frac{1}{2n} + \frac{1}{2n} = \frac{1}{n},$$ for all $x\in M\cap B_{n^2+\delta_n}$.
		
		To prove the last part of the theorem, let $(f^{(k)})_k$ be a bounded sequence in $\Lipo{M}$ converging to $f$ pointwise. For a fixed $x\in M\cap B_{n^2+\delta_n}$ we get $\hat{f}^{(k)} _{s_n}(x)\to \hat{f} _{s_n}(x)$ by the Dominated Convergence Theorem. Now it is easily seen that $S_n(f_k)(x) \to S_n(f)(x)$ pointwise.
	\end{proof}
	
	The next two results stipulate how to construct of finite-rank operators that can closely approximate the mollified functions furnished by the $S_n$ in both a Lipschitz and uniform sense.
	
	\begin{thm} \label{thm:smoothingprf}
		Let $x\in M\cap B_{n^2}$ and let $E_x\colon \R^d\to T_x$ be a linear $\nm\cdot _2$-isometry. Then there exists a neighbourhood $W$ of $0\in\R^d$, such that the map $u\mapsto\psi(x+E_x u)$ from $W$ to $M$ is open and, for all $f\in \Lipo{M}$ the function $\tilde f_x\colon W \to \R$, $$\tilde f_x(u) = S_n(f)(\psi(x+E_x u)),$$ is well-defined and $C^1$. Moreover, for every $\epsilon>0$ there exists $\delta>0$ such that $\nm{D\tilde f_x(u) - D\tilde f_x(u')}_2 \leq \epsilon$ whenever $f\in B_{\Lipo{M}}$ and $u,u'\in W$ satisfy $\nm{u-u'}_2 \leq \sqrt d \delta$.
	\end{thm}
	\begin{proof}
		By \Cref{lm:dpsiproj}, $D(\psi - I)(x)|_{T_x} = 0$. Because $\psi$ is $C^1$, there exists $V \ni x$ open in $(x+T_x) \cap M^{\delta_n} \cap \inter{B_{n^2+\delta_n}}$, such that $\Lip_{\nm\cdot _2}((\psi - I)|_V) \leq \frac{1}{2}$. Therefore the inverse of $\psi|_V$ exists and is Lipschitz, as is shown after the proof of \cite{ps:15}*{Lemma 2.7}. Let $\phi_x$ denote this inverse defined on some set $U \ni x$ that is open in $M\cap \inter{B_{n^2 +\delta_n}}$ of $x$. We put $W = E_x^{-1} (\phi_x(U)-x)$.
		
		For $f\in\Lipo{M}$ we have $$\tilde f _x(u) = \hat f  _{s_n} (\psi(x+E_x u)) - \hat f _{s_n} (0),$$ for $u\in W$. According to \Cref{lem:fhatiscinf}, $\hat f _{s_n}$ is $C^\infty$. Since $E_x$ is $C^\infty$ and $\psi$ is $C^1$ we have that $\tilde f _x$ is $C^1$. Now let $\epsilon > 0$ and by uniform continuity of $D\psi$ on compact sets choose
		\begin{align}\label{eq:delta}
		\delta \in \left(0,\frac{\epsilon s_n}{2GKL_n^3\sqrt d}\right) 
		\end{align} such that
		\begin{align}\label{eq:dpsi}
		\nm{D\psi (y) - D\psi (z)}_{op} \leq \frac{\epsilon}{2KL_n}
		\end{align}
		whenever $y,z\in M^{\delta_n}\cap \inter{B_{n^2+\delta_n}}$ satisfy $\nm{y-z}_2 \leq \sqrt d \delta$. We note that $\nm{D\psi(y)}_2\leq L_n$ when $y \in M^{\delta_n} \cap B_{n^2+\delta_n}$.
		
		Pick $f\in B_{\Lipo{M}}$ and $u,u' \in W$ satisfying $\nm{u-u'}_2\leq \sqrt d\delta$. Note that $x+E_x u, \psi(x+E_x u) \in M^{\delta_n} \cap \inter{B_{n^2+\delta_n}}$, and the same holds for $u'$. By the chain rule, $$D\tilde f _x(u) = (D\hat f _{s_n})(\psi(x+E_x u))\circ (D\psi)(x+E_x u)\circ E_x.$$ Again by \Cref{lem:fhatiscinf}, together with the fact that $E_x$ is a $\nm\cdot _2$-isometry, we obtain
		\begin{align*}
		&\nm{D\tilde f _x(u) - D\tilde f _x(u')}_2 \\ \leq& \nm{(D\hat f _{s_n})(\psi(x+E_x u)) - (D\hat f _{s_n})(\psi(x+E_x u'))}_2 \nm{(D\psi)(x+E_x u)}_2 \\ +&  \nm{(D\hat f _{s_n})(\psi(x+E_x u'))}_2 \nm{ (D\psi)(x+E_x u) - (D\psi)(x+E_x u')}_2 \\ \leq & \frac{GKL_n}{s_n} \nm{\psi(x+E_x u)-\psi(x+E_x u')}_2 L_n + KL_n \frac{\epsilon}{2KL_n} \\ \leq & \frac{GKL_n^3}{s_n} \nm{u-u'}_2 + \frac{\epsilon}{2} \leq \frac{\epsilon}{2} + \frac{\epsilon}{2} = \epsilon \tag*{by \eqref{eq:dpsi} and \eqref{eq:delta}. \qedhere}
		\end{align*}
	\end{proof}
	
	Now let $x\in M\cap B_{n^2}$ and choose some $\nm\cdot _2$-isometry $E_x\colon \R^d\to T_x$. Let $W_x$ be the neighbourhood of $0\in\R^d$ furnished by \Cref{thm:smoothingprf}. Let $C_x$ be some closed hypercube centered at $0$ and contained in $W_x$, and let $U_x = \psi(x+E_x(\inter{C_x}))$. We have that $U_x$ is open in $M$. We form the cover $\{U_x : x\in M\}$ of $M\cap B_{n^2}$ and find a finite subcover $U_1,\ldots,U_m$ with associated points $x_1,\ldots, x_m$. Let $J_n = \max\{\Lip_{\nm\cdot _2}(\phi_{x_i} |_{U_i} ): i=1, \ldots, m\}$.
	
	\begin{thm} \label{thm:finiterankop}
		Let $i\in\{1,\ldots,m\}$, and $\epsilon > 0$. There exists a finite-rank linear map $P_i \colon\Lipo{M\cap B_{n^2+\delta_n}}\to C(\overline{U_i})$ such that
		\begin{align*}
		\Lip(P_i(S_n(f))-S_n(f)|_{U_i}) &\leq KJ_n(1+\sqrt d)\epsilon \text{ and }\\ \nm{P_i(S_n(f))-S_n(f)|_{U_i}}_\infty &\leq 2KL_n\sqrt d\epsilon
		\end{align*}
		for all $f\in B_{\Lipo{M}}$. Moreover, for any bounded sequence $(f_k)_k \subseteq \Lipo{M\cap B_{n^2+\delta_n}}$ converging pointwise to $f$ we have $P_i(f_k) \to P_i(f)$ pointwise.
	\end{thm}
	\begin{proof}
		Fix $i\in\{1,\ldots,m\}$ and $\epsilon > 0$. Choose $\delta \in (0,\epsilon]$ furnished by \Cref{thm:smoothingprf} applied to $x_i, E_{x_i}$ and $\epsilon > 0$.
		
		Let $e(C_{x_i})$ be the edge length of $C_{x_i}$ and let $b$ be the least integer not smaller than $\frac{e(C_{x_i})}{\delta}$. Put $\xi = \frac{e(C_{x_i})}{b} \leq \delta$. Let $\mathcal{C}$ be the cover of $C_{x_i}$ with hypercubes of edge length $\xi$, determined by the mesh $\xi \mathbb{Z}^d$. In other words, $$\mathcal{C} = \{ C\subseteq \R^d : C \text{ is a hypercube, } V_C = C\cap \xi \mathbb{Z}^d \text{ and } \inter{C\cap C_{x_i}} \not=\emptyset\},$$ where $V_C$ is the set of vertices of $C$. 
		
		For any $f\in\Lipo{M\cap B_{n^2+\delta_n}}$, we define the function $G(f)\colon W_{x_i}\to \R$ by $G(f)(u) = f(\psi(x_i + E_{x_i}u))$. Then we define $\Lambda(G(f))\colon C_{x_i}\to\R$, $\Lambda(G(f))(u) = \Lambda(G(f),C)(u)$, where $C\in\mathcal{C}$ is such that $u\in C$. Note that if $C,C'\in \mathcal{C}$ are different hypercubes with $C\cap C'\not=\emptyset$, then $\Lambda(G(f),C)$ and $\Lambda(G(f),C')$ are equal on $C\cap C'$. Therefore $\Lambda(G(f))$ is well-defined. Finally, we define $$P_i(f)(y) = \Lambda(G(f))(E_{x_i}^{-1}(\phi_{x_i}(y)-x_i))$$ for $y\in U_i$. We may and do extend $P_i(f)$ to $\overline{U_i}$, since $P_i(f)$ is Lipschitz. We have that $P_i$ is a finite-rank linear map.
		
		Now fix $f\in B_{\Lipo{M}}$ and consider the function $g = G(S_n(f))$. By the choice of $\delta$ supplied by \Cref{thm:smoothingprf}, the function $g$ satisfies the assumptions of \Cref{lm:grid} for every $C\in\mathcal{C}$. We obtain $$\Lip_{\nm\cdot _2}((\Lambda(g) - g)|_{C_{x_i}}) \leq (1+\sqrt d)\epsilon \quad\text{and}\quad \nm{(\Lambda(g ) - g)|_{C_{x_i}} }_\infty \leq \sqrt d\delta\Lip_{\nm\cdot _2}(g).$$
		
		Pick $y,z\in U_i$. We have $$P_i(S_n(f))(y) = \Lambda(g)(E_{x_i}^{-1}(\phi_{x_i}(y)-x_i)),$$ and $$S_n(f)(y) = G(S_n(f))(E_{x_i}^{-1}(\phi_{x_i}(y)-x_i)) = g(E_{x_i}^{-1}(\phi_{x_i}(y)-x_i))$$ (and similarly for $z$). Therefore
		\begin{align*}
		&|P_i(S_n(f))(y) - S_n(f)(y) - (P_i(S_n(f))(z) - S_n(f)(z))| \\=& |(\Lambda(g) - g) (E_{x_i}^{-1}(\phi_{x_i}(y)-x_i)) - ((\Lambda(g) - g) (E_{x_i}^{-1}(\phi_{x_i}(z)-x_i)))| \\ \leq& \Lip_{\nm\cdot _2}((\Lambda(g) - g)|_{C_{x_i}}) \nm{(E_{x_i}^{-1}(\phi_{x_i}(y)-x_i)) - (E_{x_i}^{-1}(\phi_{x_i}(z)-x_i))}_2 \\ \leq& (1+\sqrt d)\epsilon \nm{\phi_{x_i} (y) - \phi_{x_i} (z)}_2 \leq KJ_n(1+\sqrt d)\epsilon \nm{y-z}.
		\end{align*}
		Therefore $\Lip(P_i(S_n(f)) - S_n(f)|_{U_i}) \leq KJ_n(1+\sqrt d)\epsilon$.
		
		For the uniform norm, we have $$\nm{P_i(S_n(f))-S_n(f)|_{U_i}}_\infty \leq \nm{(\Lambda(g ) - g)|_{C_{x_i}} }_\infty \leq \sqrt d \delta\Lip_{\nm\cdot _2}(g).$$ If $u,v\in C_{x_i}$ then 
		\begin{align*}
		|g(u) - g(v)| &= |S_n(f)(\psi(x_i + E_{x_i}u)) -  S_n(f)(\psi(x_i + E_{x_i}v))| \\&\leq \Lip(S_n(f))\nm{\psi(x_i + E_{x_i}u) - \psi(x_i + E_{x_i}v)} \\&\leq \left(1+\frac{1}{n}\right) K\nm{\psi(x_i + E_{x_i}u) - \psi(x_i + E_{x_i}v)}_2 \leq 2 KL_n\nm{u-v}_2.
		\end{align*}
		Therefore $\Lip_{\nm\cdot _2}(g) \leq 2KL_n$ and $$\nm{P_i(S_n(f))-S_n(f)|_{U_i}}_\infty \leq 2KL_n \sqrt d \delta \leq 2KL_n\sqrt d\epsilon.$$
		
		Finally, let $(f_k)_k$ be a bounded sequence in $\Lipo{M\cap B_{n^2+\delta_n}}$ converging pointwise to $f$ and let $y\in U_i$. We have that $G(f_k) \to G(f)$ pointwise. Let $C\in \mathcal{C}$ be a hypercube containing the point $u = E_{x_i}^{-1}(\phi_{x_i}(y)-x_i)$. Then $P_i(f_k)(y) = \Lambda(G(f_k))(u) = \Lambda(G(f_k), C)(u) \to \Lambda(G(f), C)(u) = P_i(f)(y)$. As $y\in U_i$ was arbitrary we have that $P_i(f_k) \to P_i(f)$ pointwise on $U_i$, and convergence on $\overline{U_i}$ follows easily.
	\end{proof}
	
	Finally we have the ingredients needed to prove \Cref{thm_main}.
	
	\begin{proof}[Proof of \Cref{thm_main}]
		As described in \cite{ps:15}*{p.~44}, it suffices to construct a sequence $(\Gamma_n)_n$ of finite-rank dual operators on $\Lipo{M}$ such that $\nm{\Gamma_n} \to 1$ and for every $x\in M$, $\Gamma_n(f)(x) \to f(x)$ uniformly in $f\in B_{\Lipo{M}}$.
		
		For a fixed $n\in \N$, consider the cover $U_1,\ldots,U_m$ of $M\cap B_{n^2}$ defined before \Cref{thm:finiterankop} and let $\alpha_i\colon M\to [0,1], i=1,\ldots,m$ be $H$-Lipschitz functions, where $H\geq 1$, forming a partition of unity subordinate to $U_1,\ldots,U_m$. Apply \Cref{thm:finiterankop} to $\epsilon := (2mnHK(1+\sqrt d)\max(L_n,J_n))^{-1}$ to obtain corresponding operators $P_i, i=1,\ldots,m$. Then $$\nm{P_i(S_n(f))-S_n(f)|_{U_i}}_\infty, \Lip(P_i(S_n(f))-S_n(f)|_{U_i}) \leq \frac{1}{mnH}$$ for all $f\in B_{\Lipo{M}}$. Define the operator $Q_n'\colon\Lipo{M}\to C(M\cap B_{n^2})$ by $$Q_n'(f)(x) = \sum_{i=1}^m \alpha_i(x) P_i(S_n(f))(x).$$ By \Cref{thm:finiterankop}, $Q_n'$ has finite rank. Then using \Cref{lm:gluing} we obtain $$\nm{Q_n'(f) - S_n(f)|_{M\cap B_{n^2}}}_\infty \leq \frac{1}{n} \text{ and } \Lip(Q_n'(f) - S_n(f)|_{M\cap B_{n^2}}) < \frac{1+mH}{mnH} \leq \frac{2}{n},$$ for all $f\in B_{\Lipo{M}}$. According to \Cref{thm:smoothingext} we have $\nm{Q_n'(f) - f|_{M \cap B_{n^2}}}_\infty < \frac{2}{n}$ and  $\Lip(Q_n'(f)) \leq \frac{2}{n}+\Lip(S_n(f)) \leq 1+\frac{3}{n}$ for all $f\in B_{\Lipo{M}}$. Define $Q_n\colon\Lipo{M}\to \Lipo{M\cap B_{n^2}},$ $$Q_n(f)(x) = Q_n'(f)(x)-Q_n'(f)(0).$$ Then $Q_n(f)(0) = 0$ for all $f\in\Lipo{M}$. Since for $f\in B_{\Lipo{M}}$, $|Q_n'(f)(0)| = |Q_n'(f)(0) - f(0)| \leq \frac{2}{n}$, we have $\nm{Q_n(f) - f|_{M \cap B_{n^2}}}_\infty < \frac{4}{n}$. Moreover, $\Lip(Q_n(f)) = \Lip(Q_n'(f)) \leq 1+\frac{3}{n}$.
		
		Finally, by \Cref{prop:flat}, there exists an operator $\Phi_n\colon\Lipo{M\cap B_{n^2}} \to \Lipo{M}$, such that $\nm {\Phi_n} \leq 1+\frac{K^2}{\log n}$ and $\Phi_n(f)$ agrees with $f$ on $M\cap B_n$ and equals $0$ outside $B_{n^2}$ for all $f\in\Lipo{M}$. Define $\Gamma_n\colon \Lipo{M}\to\Lipo{M}$ by $\Gamma_n(f) = \Phi_n(Q_n(f)).$ We have $$\nm{\Gamma_n}\leq \left(1+\frac{K^2}{\log n}\right)\left( 1+\frac{3}{n} \right) \to 1 \text{ as } n\to\infty.$$ Moreover, for a fixed $x\in M$ and all $n\geq \nm x _2$ we have $$|\Gamma_n(f)(x)-f(x)| = |Q_n(f)(x)-f(x)| \leq \frac{4}{n},$$ for all $f\in B_{\Lipo{M}}$. 
		
		To show that $\Gamma_n$ is a dual operator it suffices to show that it is $w^*$-$w^*$ continuous. By the Banach-Dieudonn\'e Theorem it suffices to show that the restriction of $\Gamma_n$ to $B_{\Lipo{M}}$ is $w^*$-$w^*$ continuous. Since $M\subseteq\R^N$, $\free{M}$ is separable, and so $B_{\Lipo{M}}$ is $w^*$-metrisable. It is therefore enough to check sequential continuity of $\Gamma_n$. The $w^*$-topology on $\Lipo{M}$ on bounded sets coincides with the topology of pointwise convergence, and so let $(f^{(k)})_k$ be a sequence in $B_{\Lipo{M}}$ converging pointwise to $f$. Then $(S_n(f^{(k)}))_k$ is a bounded sequence converging pointwise to $S_n(f)$ by \Cref{thm:smoothingext}. By \Cref{thm:finiterankop}, $P_i(S_n(f^{(k)})) \to P_i(S_n(f))$ pointwise for every $i=1,\ldots,m$. Then clearly $Q_n'(f^{(k)}) \to Q_n'(f)$ pointwise which easily implies that $Q_n(f^{(k)})\to Q_n(f)$ and $\Gamma_n(f^{(k)})\to \Gamma_n(f)$.
	\end{proof}
	
	\section{Open problems}\label{sect_open_problems}
	
	We conclude with a few remarks about open problems. We don't know if Theorem \ref{thm_main} can be generalised to include non-$C^1$-submanifolds. In particular, we don't know the answer in the case where $M$ is the unit sphere of an arbitrary norm on $\R^N$, $N \geq 3$.  We can answer this in the case $N=2$ by appealing to existing results.
	
	Let $\nm\cdot, \nmt\cdot$ be two norms on $\R^2$ (neither necessarily differentiable), let $M$ be the unit sphere of $\nm\cdot$ endowed with the metric $d$ induced by$\nmt\cdot$ and choose a distinguished point $x_0\in M$. We will denote by $\partial D$ the boundary of a subset $D\subseteq\R^2$.
	
	\begin{thm}
		The space $\free{M}$ has the MAP.
	\end{thm}
	
	\begin{proof}
		Denote by $B$ be closed unit ball of $\nm\cdot$. Let $D = 2B \setminus \inter{B}$, where $2B := \{2x : x\in B\}$. Equip $D$ with the metric $d$ and let $x_0$ be the distinguished point for $D$ as well. By \cite{b:73}*{Theorem 5}, there exists a map $h\colon \R^2 \to B$, such that $h(x)=x$ for all $x\in B$ and $h$ is contractive, i.e.~$d(h(x),h(y)) \leq d(x,y)$ for all $x,y\in \R^2$. Let $h|_D$ be the restriction of $h$ to $D$ and note the the range of $h|_D$ is $M$. By the lifting property for Lipschitz maps, $h|_D$ induces a norm-one projection $H\colon \free{D}\to\free{M}$. Therefore it suffices to show that $\free{D}$ has the MAP. By \cite{ps:15}*{Corollary 2.4} it suffices to show that $D$ is locally downwards closed, i.e.~given $x\in\partial D$, there exists open $U\ni x$ and $v\not= 0$ such that $y-tv\in\inter{D}$ whenever $y\in U\cap D, y-tv\in U$ and $t>0$. This is straightforward to see: given $x\in\partial D$, set $v=x$ if $x\in\partial (2B)$ and $v=-x$ if $x\in\partial B$, and let $U$ be a sufficiently small open ball having centre $x$.
	\end{proof}
	
	
	
	\bibliography{manifolds_map}
	
\end{document}